\documentclass[a4paper,12pt, reqno]{amsart}

\usepackage{amsfonts}
\usepackage{amsmath}
\usepackage{amssymb}
\usepackage{mathrsfs}
\usepackage{hyperref}
\usepackage{graphicx}
\usepackage{comment}
\usepackage{csquotes}
\usepackage[usenames]{color}

\numberwithin{equation}{section}
\theoremstyle{plain}
\newtheorem{thm}{Theorem}[section]

\newtheorem{lemma}[thm]{Lemma}

\theoremstyle{definition}
\newtheorem{defi}[thm]{Definition}

\begin{document}

\title[A comparison principle for nonlinear parabolic equations ]
{A comparison principle for nonlinear parabolic equations with nonlocal source and gradient absorption 
}

\author[Z. Amirzhankyzy]{Zhaniya Amirzhankyzy}
\address{
  Zhaniya Amirzhankyzy:
    \endgraf
  SDU University, Kaskelen, Kazakhstan
\endgraf
    {\it E-mail address} {\rm zhaniyaamirzhankyzy@gmail.com}
  }
\author[N. Yessirkegenov]{Nurgissa Yessirkegenov}
\address{
  Nurgissa Yessirkegenov:
  \endgraf
    KIMEP University, Almaty, Kazakhstan
  \endgraf
  and
  \endgraf
        Institute of Mathematics and Mathematical Modeling, Almaty, Kazakhstan
    \endgraf
  {\it E-mail address} {\rm nurgissa.yessirkegenov@gmail.com}
  }

\thanks{This research is funded by the Committee of Science of the Ministry of Science and Higher Education of the Republic of Kazakhstan (Grant No. AP23490970)}


\begin{abstract}
  This paper investigates the initial-boundary value problem for a nonlinear parabolic equation involving the \(p\)-Laplacian operator, nonlocal source terms, gradient absorption, and various nonlinearities:
\[
\frac{\partial u}{\partial t} - \text{div}(|\nabla u|^{p-2} \nabla u ) = \alpha |u|^{k-1}u \int_\Omega |u|^s \, dx - \beta |u|^{l-1}u |\nabla u|^q + \gamma u^m + \mu |\nabla u|^r - \nu |u|^{\sigma-1}u,
\]
where \( \Omega \) is a bounded domain in \( \mathbb{R}^N \), \( N \geq 1 \), with a smooth boundary \( \partial \Omega \). The parameters satisfy \( \alpha, l, \sigma > 0 \), \( \beta, \nu \geq 0 \), \( k, m, s \geq 1 \), \( r \geq p - 1 \geq \frac{p}{2} \), and \( \gamma, \mu \in \mathbb{R} \).

We establish a comparison principle for this problem. Using this principle, we derive blow-up results as well as global-in-time boundedness of solutions. Our results extend and unify previous studies in the literature.
\end{abstract}

\maketitle

\section{Introduction}
Let us consider the following initial-boundary value problem:
  \begin{equation}
    \left\{
\begin{aligned}
& \frac{\partial u}{\partial t} - \text{div}(|\nabla u|^{p-2} \nabla u ) = \\
& \quad {\alpha} |u|^{k-1}u \int_\Omega |u|^s dx -\beta |u|^{l-1}u |\nabla u|^{q}+\gamma u^m + \mu |\nabla u|^{r} \\
& \quad - \nu |u|^{\sigma-1}u, \quad \;\;  x \in \Omega, \;t > 0, \\
& u = 0, \quad  \quad  \quad  \quad \; \;\;\;  x \in \partial \Omega ,  \, t > 0,\\
& u(x,0) = u_0 (x) , \quad x \in  \Omega, 
\end{aligned}
\right.
\label{eq_1}
 \end{equation}
where $\Omega $ is a bounded domain of $\mathbb{R}^{N},N \geq 1$ with a smooth boundary $\partial \Omega$, $\alpha$, $l$, $\sigma > 0$, $\beta$,$\nu$ $\geq 0$, $k$,$m$,$s$ $\geq$ $1$, $r$ $\geq$ $p-1$ $\geq$ $\frac{p}{2}$, $\gamma$, $\mu$ $\in$ $\mathbb{R}$. Throughout this paper we suppose that the initial data $u_0$ satisfy
\begin{equation}
    u_0 (x) \geq 0, u_0(x) \not\equiv 0, u_0 \in W^{1,\infty}(\Omega).
\end{equation}
A wide variety of nonlinear parabolic problems with integral, gradient, and absorption terms have been investigated in recent decades. Researchers have addressed different configurations of such terms depending on the parameters and structure of the equation, including the presence of nonlocal sources, degenerate diffusion, and gradient-dependent nonlinearities. Many results are known regarding existence, uniqueness, and blow-up behavior under various settings, which serve as a foundation for analyzing the general model we propose in this work.

For example, the classification of blow-up behavior for equations with nonlinear gradient terms was studied in \cite{ZL16}, while a broader qualitative analysis of nonnegative solutions with variable exponents was provided in \cite{CH18}. Classical monographs such as \cite{QS07} and papers like \cite{Sou01} also address the role of absorption, boundary effects, and multiple nonlinearities in related problems.

In particular, several special cases of problems involving nonlinear diffusion, gradient absorption, and nonlocal sources have been analyzed separately in the literature.For example, when all reaction and source terms are absent, leaving only the gradient absorption term on the right-hand side, the well-posedness and blow-up phenomena were studied by Attouchi \cite{Att12}. In this case, the model reduces to the following simpler form:
$$\frac{\partial u}{\partial t} - \text{div}(|\nabla u|^{p-2} \nabla u ) = |\nabla u|^r.$$
This formulation emphasized the effects of nonlinear diffusion and gradient absorption terms while excluding nonlocal sources and reaction terms.

Several additional works extend the theory of global existence and blow-up criteria for nonlinear parabolic problems with similar structures. The global solution in the presence of a nonlinear gradient term was studied in \cite{Zha13}, and the interplay between nonlocal terms and logarithmic nonlinearities was analyzed in \cite{YY18}. In the context of quasilinear equations with absorption and nonlinear boundary conditions, a comprehensive analysis was carried out in \cite{AMZ14}.

Furthermore, the comparison principle was later extended to a more general case when $\gamma$ $=$ $\mu$ $=$ $\nu$ $=$ $0$, where the integral, gradient absorption and nonlocal source term remained, as was discussed in \cite{CH20}. The following problem  
$$\frac{\partial u}{\partial t} - \text{div}(|\nabla u|^{p-2} \nabla u )={\alpha} |u|^{k-1}u \int_\Omega |u|^s dx -\beta |u|^{l-1}u |\nabla u|^{q}$$
was investigated by numerous researchers over the past decade (see \cite{BL08, IL12, IL14, ILS17}), particularly within the framework of viscosity solutions, under the assumption that $\alpha$ $=$ $l$ $=$ $0$ and for $p$ $>$ $1$. A significant number of these studies focused on deriving decay estimates, establishing finite propagation speed, and analyzing long-term behavior. 

More recently, further insights were obtained for gradient source-type nonlinearities and blow-up thresholds in problems similar to ours. These include the study of p-Laplacian models with gradient terms in \cite{Din20}, extinction behavior with strong absorption in \cite{FWL13}, and foundational studies like \cite{Dia01}, which provide qualitative frameworks for nonlinear parabolic dynamics.

Notable contributions to the case where $\beta = 0$ can be found in the works of \cite{LX03} and \cite{FZ14}, which examined the existence and nonexistence of global weak solutions. Many classical theories, however, no longer hold, as nonlocal models are more relevant to practical applications than local ones. As a result, these models pose additional challenges and introduce greater complexity.  

Another important case was considered by Zhang and Li \cite{ZL13}, where $\alpha$ $=$ $\beta$ $=$ $\nu$ $=$ $0$, leading to an equation that involves only local nonlinear diffusion, reaction, and gradient absorption terms:
$$\frac{\partial u}{\partial t} - \text{div}(|\nabla u|^{p-2} \nabla u ) = \gamma u^m + \mu |\nabla u|^{r}.$$
Gradient-based blow-up phenomena, especially with degenerate diffusion, were also explored in \cite{CW89}, offering valuable energy estimates and blow-up rates that influence our setting.

We also refer to \cite{RS18} and \cite{RY22} for similar investigations using comparison principles on graded Lie groups.

Thus, we observe that different cases were considered separately for the integral term, the gradient term, and the nonlocal source term, with the existence of a weak solution proven for each case. This led us to pose the following question: Does the comparison principle hold for a problem that incorporates all these terms simultaneously?

Specifically, we aim to investigate the following scenario: $$\frac{\partial u}{\partial t}- \text{div}(|\nabla u|^{p-2}\nabla u) = \alpha u^k \int_\Omega u^s dx - \beta u^l |\nabla u|^q + \gamma u^m + \mu |\nabla u|^r-\nu |u|^{\sigma -1}u.$$ It turns out that the comparison principle does indeed work for this equation.

Our goal in this note is to establish a comparison principle for the initial-boundary value problem involving a nonlinear parabolic equation with a $p$-Laplacian operator, a nonlocal source, gradient absorption, and various nonlinear terms. We extend and refine previous results, providing new insights into the interplay between diffusion, reaction, and absorption effects. Furthermore, we discuss applications of this principle for global existence, blow-up phenomena, and the qualitative behavior of weak solutions.

The structure of this paper is as follows. In Section 2, we present and prove the comparison principle for the problem (\ref{eq_1}). In Section 3, we apply this principle to analyze conditions leading to global existence or blow-up in finite time, depending on the signs and relationships among the parameters  $\alpha$, $\beta$, $\gamma$, $\mu$, $\nu$, and the exponents  $k$, $s$, $l$, $q$, $m$, $r$, $\sigma$.

\section{comparison principle for the problem (\ref{eq_1})}

Prior to defining a weak solution, we denote
    $$Q_T=\Omega \times (0,T), \partial Q_T = \{ \partial \Omega \times [0,T]\} \cup \{\overline{{\Omega}} \times \{0\}\} \  \text{and} \ |\Omega| = meas(\Omega).$$
Now, let us introduce the definition of a weak solution of problem (\ref{eq_1}).
\newtheorem{definition}{Definition}
\begin{defi}
Let $\Tilde{\alpha}$ $=$ $\max\{ p, k+s, l+q, m, r, \sigma \}$.
A nonnegative function $u(x,t)$ is called a weak sub- (super-) solution of (\ref{eq_1}) on $Q_T$ if it satisfies
\begin{enumerate}
    \item \( u \in C ([0,T) \times \overline{\Omega}) \cap L^{\Tilde{\alpha}} (0,T; W^{1,\Tilde{\alpha}}_0(\Omega))\) and \( \frac{\partial u}{\partial t} \in L^{2}(Q_T)\).
    \item For every non-negative test-function \(\phi \in C (\overline{Q_T}) \cap L^{p}(0,T; W_0^{1,p}(\Omega)),\) we have
\begin{align}
&\int_{Q_T} \left( \partial_t u \phi + |\nabla u|^{p-2} \nabla u \cdot \nabla \phi \right) \, \mathrm{d}x \, \mathrm{d}t \notag = \int_{Q_T} \Big( \alpha |u|^{k-1} u \int_\Omega |u|^{s} \, \mathrm{d}x - \beta |u|^{l-1} u |\nabla u|^{q} \notag \\
&\quad + \gamma u^{m} + \mu |\nabla u|^{r} - \nu |u|^{\sigma-1} u \Big) \phi \, \mathrm{d}x \, \mathrm{d}t.
\label{eq_2}
\end{align}
\item 
\begin{equation}
        u(x, 0) = u_0(x) \, \;\text{in} \;\, \Omega, \quad \text{and} \quad u = 0 \; \text{on} \,\; \partial \Omega \times (0, T).
\label{eq_3}
\end{equation}
\end{enumerate}
Moreover, if we replace \enquote{$=$} in (\ref{eq_2}) and (\ref{eq_3}) by \enquote{$\leq$}$($\enquote{$\geq$}$)$
 then the corresponding solution is called a sub- (super-)solution.
\label{definition_1}
\end{defi}
\begin{thm}\label{thm_comp_princ}
    Let $\alpha$, $l$, $\sigma > 0$, $\beta$, $\nu$ $\geq 0$, $r \geq p-1 \geq \frac{p}{2}$, $q \geq \frac{p}{2}$, $k,m,s\geq 1$. Suppose that $u,v \in L^{\infty}(0,T; W_0^{1,\infty}(\Omega))$ are weak sub- and sup- solutions of (\ref{eq_1}){}{} respectively. If $u$ and $v$ are locally bounded, then $ u \leq \, v$ a.e. in  $\Omega$.
    \label{Theorem 2.1}
\end{thm}
The following algebraic lemma (see e.g. [\cite{Att12}, Lemma 2.1]) serves as the primary foundation for the proof of the comparison principle.

\begin{lemma}\label{lem_algeb_ineq}
    Let \(\tilde{\sigma} > 1\). For all \(\vec{a}, \vec{b} \in \mathbb{R}^N \), then we have
$$ \Big\langle|\vec{a}|^{\tilde{\sigma} - 2} \vec{a} - |\vec{b}|^{\tilde{\sigma} - 2} \vec{b}, \vec{a} - \vec{b} \Big \rangle \geq  \frac{4}{\tilde{\sigma}^{2}} \Big||\vec{a}|^\frac{\tilde{\sigma} -2}{2} \vec{a} - |\vec{b}|^\frac{\tilde{\sigma} -2}{2}  \vec{b} \Big|^{2}.$$ 
\end{lemma}

\begin{proof}[Proof of Theorem \ref{thm_comp_princ}]
Let $\phi$ $=$ $\max \{u-v, 0\}$, then $\phi(x, 0)=0$ and $\left.\phi(x, t)\right|_{x \in \partial \Omega}=0$. According to the definitions of sub- and super-solutions, by selecting \(\phi\) as the test function, it follows that for any $\tau$ $\in$ $(0,T)$, the following holds:
\begin{equation}
\begin{split}
\int_0^\tau \int_\Omega \partial_t \phi \, \phi \, dx \, dt \leq&\underbrace{ -\int_0^\tau \int_{\{\phi(.,t) > 0\}} \left[|\nabla u|^{p-2}\nabla u - |\nabla v|^{p-2}\nabla v \right] \nabla\phi \, dx \, dt }_{A_{1}}\\ 
 &\underbrace{+\alpha\int_0^\tau \int_{\{\phi(.,t) > 0\}} \left[u^k \int_\Omega u^{s} \, dx - v^{k} \int_\Omega v^{s} \, dx\right] \phi \, dx \, dt}_{B}\\
  &\underbrace{-\beta \int_0^\tau \int_{\{\phi(.,t) > 0\}} \left[u^{l} |\nabla u|^{q} - v^{l} |\nabla v|^{q}\right] \phi \, dx \, dt}_{H}\\
&\underbrace{+\gamma \int_0^\tau \int_{\{\phi(.,t) > 0\}} (u^{m} - v^{m})\phi \, dx \, dt}_{S} \\
 &\underbrace{+\mu\int_0^\tau \int_{\{\phi(.,t) > 0\}} \left[|\nabla u|^{r}-|\nabla v|^{r}\right]\phi \, dx \, dt }_{G}\\
 &\underbrace{-\nu \int_0^\tau \int_{\{\phi(.,t) > 0\}} (|u|^{\sigma - 1}u - |v|^{\sigma - 1}v)\phi \, dx \, dt}_{A_{2}}.
\end{split}
\label{eq_4}
\end{equation}

Let us now estimate the terms $A_1$,$B$,$C$,$S$,$G$, and $A_2$ that appear in the inequality (\ref{eq_4}).

By applying Lemma \ref{lem_algeb_ineq} , we obtain the following estimate for $A_1$:
\begin{equation}
    \begin{aligned}
        A_1&=-\int_0^\tau \int_{\{\phi(.,t) > 0\}} \left[|\nabla u|^{p-2}\nabla u - |\nabla v|^{p-2}\nabla v \right]\nabla \phi \, dx \, dt \\
        & \leq - \frac{4}{p^{2}} \int_0^\tau \int_{\{\phi(.,t) > 0\}} \left||\nabla u|^{\frac{p-2}{2}}\nabla u - |\nabla v|^{\frac{p-2}{2}}\nabla v \right|^{2}  \, dx \, dt.
    \end{aligned}
    \label{2.4}
\end{equation}

To evaluate the variable $B$, we transform the expression and decompose $B$ into two components, assigning the first component as $D_1$ and the second one as $D_2$.
\begin{equation}
\begin{split}
    B &= \alpha \int_0^\tau \int_{\{\phi(.,t) > 0\}}\left[u^{k} \int_\Omega u^{s} \, dx - v^{k} \int_\Omega v^{s} \, dx \right] \, \phi \, dx \, dt \\
    &=\alpha \underbrace{\int_0^\tau \int_{\{\phi(.,t) > 0\}} \left[u^{k} \phi \left(\int_\Omega u^{s} \, dx - \int_\Omega v^{s} \, dx \right)\right] \, dx \, dt}_{D_1} \\
    &+\alpha \underbrace{\int_0^\tau \int_{\{\phi(.,t) > 0\}} \left[\phi \int_\Omega v^{s} \, dx \, (u^{k} - v^{k})\right] \, dx \, dt.}_{D_2}
\end{split}
\end{equation}
Now, let us evaluate each of the components separately.
\begin{align*}
    D_1 = \int_0^\tau \int_{\{\phi(.,t) > 0\}} u^{k} \phi \left( \int_\Omega (u^{s} - v^{s}) \, dx \right) \, dx \, dt.
\end{align*}
To assess the term $D_1$, we apply the mean value theorem (MVT) to obtain

\begin{equation}
\begin{aligned}
    D_1 = &\int_0^\tau \int_{\{\phi(.,t) > 0\}} u^{k} \phi \left( \int_{u \geq v} (u^{s} - v^{s})\, dx + \int_{u < v} (u^{s} - v^{s})\, dx \right)\, dx \, dt \\
    &\leq \int_0^\tau \int_{\{\phi(.,t) > 0\}} u^{k} \phi \left( \int_{u \geq v} (u^{s} - v^{s})\, dx \right) \, dx \, dt \\
    &\overset{\text{MVT}}{=} \int_0^\tau \int_{\{\phi(.,t) > 0\}} u^{k} \phi \int_{u \geq v} s \, w^{s-1} (u - v)\, dx \, dt  \quad \text{where} \  w \in (v,u) \\
    &\leq \int_0^\tau \int_{\{\phi(.,t) > 0\}} u^{k} \phi^{2} \, s \left( \int_{u \geq v} u^{s-1}\, dx \right)\, dx \, dt \\
    &\leq s \int_0^\tau \int_{\{\phi(.,t) > 0\}} u^{k+s-1} \phi^{2}\, dx\, dt \left( \int_{u \geq v} \, dx \right) \\
    &\leq s|\Omega| \|u\|^{k+s-1}_{L^{\infty}} \int_0^\tau \int_{\{\phi(.,t) > 0\}} \phi^{2}\, dx\, dt.
\end{aligned}
\end{equation}
Next, we proceed with the evaluation of $D_2$ in a similar manner
\begin{equation}
\begin{aligned}
    D_2 =&\int_0^\tau \int_{\{\phi(.,t) > 0\}} \left[\phi \int_\Omega v^{s} \, dx (u^{k} - v^{k})\right] dx \, dt \\
    &= \int_0^\tau \int_{\{\phi(.,t) > 0\}} \left[\phi \int_{u \geq v} v^{s} \, dx (u^{k} - v^{k})\right] dx \, dt \\
    &+\int_0^\tau \int_{\{\phi(.,t) > 0\}} \left[\phi \int_{u < v} v^{s} \, dx (u^{k} - v^{k})\right] dx \, dt \\
    &\leq \int_0^\tau \int_{\{\phi(.,t) > 0\}} \left[\phi \, \max \, \{u - v\} \int_{u \geq v} v^{s} \, dx \, \frac{u^{k} - v^{k}}{u - v}\right] dx \, dt \\
    &\overset{\text{MVT}}{\leq} k \int_0^\tau \int_{\{\phi(.,t) > 0\}} \phi^{2}k c^{k-1} \left(\int_\Omega v^{s} \, dx\right) dx \, dt \quad \text{where} \  c \in (v,u) \\
    &\leq k \int_0^\tau \int_{\{\phi(.,t) > 0\}} \phi^{2} u^{k-1} \left(\int_\Omega v^{s} \, dx\right) dx \, dt \\
    &\leq k ||u||^{k+s-1}_{L^{\infty}} \, |\Omega| \int_0^\tau \int_{\{\phi(.,t) > 0\}} \phi^{2} \, dx \, dt.
\end{aligned}
\end{equation}

Then, we deduce that 
\begin{equation}
    B =\alpha D_1+ \alpha D_2 \leq \alpha C \int_0^\tau \int_{\{\phi(.,t) > 0\}}  \phi^{2} \, dx \, dt.
\label{2.8}
\end{equation}
Let us now evaluate the term $H$, by decomposing it into the following components:
\begin{equation}
\begin{aligned}
    &H = -\beta \int_0^\tau \int_{\{\phi(.,t) > 0\}} \left[u^{l} |\nabla u|^{q} - v^{l} |\nabla v|^{q}\right] \phi \, dx \, dt \\
    &= - \beta \int_0^\tau \int_{\{\phi(.,t) > 0\}} \left[u^{l} |\nabla u|^{q} - u^{l} |\nabla v|^{q}\right] \phi \, dx \, dt \\
    &- \beta \int_0^\tau \int_{\{\phi(.,t) > 0\}} \left[u^{l} |\nabla v|^{q} - v^{l} |\nabla v|^{q}\right] \phi \, dx \, dt \\
    &= - \beta \int_0^\tau \int_{\{\phi(.,t) > 0\}} \left[u^{l} (|\nabla u|^{q} - |\nabla v|^{q})\right] \phi \, dx \, dt \\
    &- \beta \int_0^\tau \int_{\{\phi(.,t) > 0\}} |\nabla v|^{q} (u^{l} - v^{l}) \phi \, dx \, dt \\
    & \leq - \beta \int_0^\tau \int_{\{\phi(.,t) > 0\}} \left[u^{l} (|\nabla u|^{q} - |\nabla v|^{q})\right] \phi \, dx \, dt, 
\end{aligned}
\label{2.12}
\end{equation}
where $\beta \geq 0$.
By using Young's inequality, we have 
\begin{equation*}
    \begin{aligned}
        H & \leq C(\epsilon)\int_0^\tau \int_{\{\phi(.,t) > 0\}} (-\beta u^{l}\phi)^{2} dx \ dt + C \epsilon \int_0^\tau \int_{\{\phi(.,t) > 0\}} (|\nabla u|^{q}-|\nabla v|^{q})^{2} dx \ dt \\
        & \leq C(\epsilon) \beta ^{2} ||u||^{2l}_{L^\infty}\int_0^\tau \int_{\{\phi(.,t) > 0\}} \phi^{2} dx \ dt + C \epsilon \int_0^\tau \int_{\{\phi(.,t) > 0\}} (|\nabla u|^{q}-|\nabla v|^{q})^{2} dx \ dt 
    \end{aligned}
\end{equation*}
 
Now let us define $h(s)= s^\frac{2q}{p}$ for $s\geq 0$. Since $q\geq p-1 \geq \frac{p}{2},$ we have $h'(s)= \frac{2q}{p}s^\frac{2q-p}{p}$. By the mean value theorem, it follows that
\begin{equation}
    \left||\nabla u|^{q}- |\nabla v|^{q} \right |^{2} \leq Ch'(\theta)^{2} \left ||\nabla u|^\frac{p}{2}-|\nabla v|^\frac{p}{2} \right |^{2},
    \label{eq_2.10}
\end{equation} 
for some $0\leq \theta \leq \text{max}(|\nabla u|^\frac{p}{2},|\nabla v|^\frac{p}{2}).$
 
A straightforward calculation shows that
 $$\left ||\nabla u|^\frac{p}{2}-|\nabla v|^\frac{p}{2} \right |^{2} \leq \left | |\nabla u|^\frac{p-2}{2}\nabla u- |\nabla v|^\frac{p-2}{2} \nabla v\right|^{2}.$$

Using the MVT we arrive at the following expression for the term $H$.
 \begin{equation}
     \begin{aligned}
         H &\leq C \epsilon \int_0^\tau \int_{\{\phi(.,t) > 0\}}\left | |\nabla u|^\frac{p-2}{2}\nabla u- |\nabla v|^\frac{p-2}{2} \nabla v\right|^{2} \, dx \, dt\\
        &+ C(\epsilon)\int_0^\tau \int_{\{\phi(.,t) > 0\}} \phi^{2} dx \, dt.
     \end{aligned}
     \label{2.11}
 \end{equation}

To estimate the following term $S$, we applied the mean value theorem (MVT).
 We have
\begin{equation}
\begin{split}
    S=\gamma \int_0^\tau \int_{\{\phi(.,t) > 0\}} (u^{m} - v^{m}) \, \phi \, dx \, dt 
    &= \gamma \int_0^\tau \int_{\{\phi(.,t) > 0\}} \frac{u^{m} - v^{m}}{u - v} (u - v) \, \phi \, dx \, dt \\
    &\overset{\text{MVT}}{\leq}\gamma m \int_0^\tau \int_{\{\phi(.,t) > 0\}} c^{m-1} \, \phi^{2} \, dx \, dt  \quad \text{where} \ c \in (v,u)\\
    &\leq \gamma m \int_0^\tau \int_{\{\phi(.,t) > 0\}} u^{m-1} \phi^{2} \, dx \, dt \\
    &\leq \gamma m \|u\|^{m-1}_{L^{\infty}} \int_0^\tau \int_{\{\phi(.,t) > 0\}} \phi^{2} \, dx \, dt.
\end{split}
\label{2.12}
\end{equation}
To estimate the term $G$, we first use Young’s inequality, followed by the application of the same argument as in equation~\eqref{eq_2.10}, taking into account the condition for $r$ and $p$:
$$ r \geq p-1 \geq \frac{p}{2}.$$ 
\begin{equation}
    \begin{aligned}
        G& = \mu \int_0^\tau \int_{\{\phi(.,t) > 0\}} \left[|\nabla u|^{r} - |\nabla v|^{r}\right] \, \phi \, dx \, dt \\
        &\overset{\text{Young}}{\leq} \mu \left[\epsilon \int_0^\tau \int_{\{\phi(.,t) > 0\}} \left||\nabla u|^{r} -|\nabla v|^{r} \right|^{2} \, dx \, dt + 
        C(\epsilon) \int_0^\tau \int_{\{\phi(.,t) > 0\}} \phi^{2} \, dx\, dt \right] \\
        &\leq \mu \left[C \ \epsilon \int_0^\tau \int_{\{\phi(.,t) > 0\}} \left||\nabla u| ^ \frac{p-2}{2} \nabla u - |\nabla v| ^ \frac{p-2}{2} \nabla v\right|^{2} \, dx\, dt + C (\epsilon)\int_0^\tau \int_{\{\phi(.,t) > 0\}} \phi^{2} \, dx \, dt \right].
    \end{aligned}
\label{2.13}
\end{equation}
 
For the term $A_2$, using the fact that
\begin{equation}
\left\{
\begin{array}{l}
|u|^{\sigma - 1} u - |v|^{\sigma - 1} v = u^{\sigma} - v^{\sigma} > 0, \ \text{if} \ u > v > 0, \\
|u|^{\sigma - 1} u - |v|^{\sigma - 1} v = u^{\sigma} + |v|^{\sigma} > 0, \ \text{if} \ u > 0 > v, \\
|u|^{\sigma - 1} u - |v|^{\sigma - 1} v = - |u|^{\sigma} + |v|^{\sigma} > 0, \ \text{if} \ 0 > u > v,
\end{array}
\right.
\label{2.7}
\end{equation}
it follows that $A_2 \leq 0$. Therefore,
\begin{equation}
    -\nu \int_0^\tau \int_{\{\phi(.,t) > 0\}} (|u|^{\sigma - 1}u - |v|^{\sigma - 1}v)\phi \, dx \, dt \leq 0.
    \label{2.15}
\end{equation}
Combining (\ref{2.4}), (\ref{2.8}), (\ref{2.11}), (\ref{2.12}), (\ref{2.12}), (\ref{2.13}), and (\ref{2.15}), and choosing $\epsilon$ small enough, we obtain
$$\int_\Omega \phi^{2} (\tau)dx \leq \int_\Omega \phi^{2} (0) dx + C\int_0^\tau \int_{\Omega} \phi^{2}dx \, dt, $$
where $C = C(\alpha, \beta, k, s,l,q, \gamma, \mu, \epsilon,m,p,r, \nu, \sigma,||u||_{L^{\infty}},\max\{|\nabla u|^{\frac{p}{2}},|\nabla v|^{\frac{p}{2}}\})$.

Applying Gronwall's inequality we conclude that for any $t \in (0,T)$, $$\int_\Omega \phi ^{2} dx = 0.$$ It follows that $\phi = 0$ a.e. $x \in \Omega$, i.e. $u \leq v$ a.e. $(x,t) \in Q_T$.

\end{proof}
\section{some applications to comparison principle for nonlinear parabolic equations with nonlocal source and gradient absorption}
\subsection{Blowup in finite time}
\begin{thm}\label{blow_thm}
Suppose that $\gamma > 0$, $r \geq p-1 \geq \frac{p}{2}$, $m > \max \{ p-1, r, \sigma \}$, $q \geq \frac{p}{2}$, $q+l>1$, $k,s\geq 1$, $k+s > \max \{p-1, q+l, r \}$  and the initial data $u_0$ is large enough, then $L^{\infty}$  finite-time blowup occurs for the problem (\ref{eq_1}). 
\end{thm}
\begin{proof}[Proof of Theorem \ref{blow_thm}]
We study the following unbounded function defined on $[t_0, \frac{1}{\delta})$ $\times$ $\mathbb{R}^{N}$
$$
v(x,t) = \frac{1}{(1 - \delta t)^{\tilde{k}}} V\left(\frac{|x|}{(1 - \delta t)^{\tilde{r}}}\right),
$$
where the function $V(y)$ is given by:
$$V(y)=\left(1+\frac{A}{\lambda}-\frac{y^{\lambda}}{\lambda A^{\lambda -1}} \right), \quad y\geq0, \quad \lambda = \frac{p}{p-1},$$
with $A$,$\Tilde{k}$,$\Tilde{r}$,$\delta$ $>0$, to be specified later.
We assume that for all $t \in [t_0,\frac{1}{\Tilde{ \delta}})$, the following condition holds:
$$\text{supp}(v(\cdot,t))=\overline{B(0,R(1-\delta t)^{\beta})}\subset \overline{B(0,R(1-\delta t_0)^{\beta})}\subset \Omega, $$
where $R=(A^{\lambda -1}(A+\lambda))^{\frac{1}{\lambda}}$ represents the unique root of $V(y).$ 

It is evident that $v(x,t)$ blows-up in finite time. Through straightforward calculations, the function $V(y)$ fulfills the following conditions
\begin{equation}
\begin{cases}
1 \leq V(y) \leq 1+\frac{A}{\lambda}, \quad -1 \leq V^{\prime}(y) \leq 0 & \text{for } 0 \leq y \leq A, \\
0 \leq V(y) \leq 1, \quad -\left(\frac{R}{A}\right)^{\lambda-1} \leq V^{\prime}(y) \leq -1 & \text{for } A \leq y \leq R, \\
(p-1)\left|V^{\prime}(y)\right|^{p-2} V^{\prime \prime}(y) + \frac{N-1}{y}\left|V^{\prime}(y)\right|^{p-2} V^{\prime}(y) = -\frac{N}{A} & \text{for } 0 < y < R.
\end{cases}
\end{equation}
Let us demonstrate that \( v(x, t) \) serves as a sub-solution of problem (\ref{eq_1}) by defining $ y = |x| (1 - \delta t)^{-\tilde{r}}$.

We consider $P(v(x, t))$ as
$$
P(v(x,t)) = \frac{\partial v}{\partial t}-\Delta_{p}v -\alpha v^{k} \int_\Omega v^{s} dx +\beta v^{l} |\nabla v|^{q} - \gamma v^{m} - \mu|\nabla v|^{r} + \nu v^{\sigma}. 
$$
Our next step is to prove that
$$
P(v(x,t)) \leq 0.
$$
Now, let us choose positive parameters $k, \tilde{r}, A,$ and $\delta$ for our problem.

If $ p(k + s - 1) + N(p - 2) > 0 $ and $r(k + s - 1) + N(r - 1) > 0$, then we choose
$$
  \tilde{r} < \min \left\{
    \frac{k + s - l - q}{q(k + s - 1) + N(l + q - 1)},
    \frac{k + s + 1 - p}{p(k + s - 1) + N(p - 2)},
    \frac{k + s - r}{r(k + s - 1) + N(r - 1)}
  \right\}.
$$
If $ p(k + s - 1) + N(p - 2) \leq 0 $, we take
  $$
  \tilde{r} < \min \left\{
    \frac{k + s - l - q}{q(k + s - 1) + N(l + q - 1)},
    \frac{k + s - r}{r(k + s - 1) + N(r - 1)}
  \right\}.
  $$

If $ r(k + s - 1) + N(r - 1) \leq 0 $, we have
  $$
  \tilde{r} < \min \left\{
    \frac{k + s - l - q}{q(k + s - 1) + N(l + q - 1)},
    \frac{k + s + 1 - p}{p(k + s - 1) + N(p - 2)}
  \right\}.
  $$

If neither condition is positive, then we set
  $$
  \tilde{r} < \frac{k + s - l - q}{q(k + s - 1) + N(l + q - 1)},
  $$
  and define
  $$
  \tilde{k} := \frac{N\tilde{r} + 1}{k + s - 1}, \quad A > \frac{\tilde{k}}{\tilde{r}}, \quad \delta < \frac{\gamma}{\tilde{k}\left(1 + \frac{A}{\lambda}\right)}.
  $$

In the case $0\leq y \leq A$, we decompose the entire expression into two terms and evaluate each separately as follows:
 \begin{equation}
\begin{aligned}
 P(v(x,t)) = & \underbrace{\frac{\partial v}{\partial t}-\Delta_pv -\gamma v^{m} - \mu|\nabla v|^{r} + \nu v^{\sigma}}_{P_1(v(x,t))} \underbrace{-\alpha v^{k} \int_\Omega v^{s} dx +\beta v^{l} |\nabla v|^{q}}_{P_2(v(x,t))} \leq 0.
\end{aligned}
\end{equation}
We first estimate $P_1(v(x,t))$ as follows:
 $$
\begin{aligned}
P_1(v(x, t))= & \frac{\delta \left(\tilde{k} V+\tilde{r} y V^{\prime}\right)}{(1-\delta t)^{\tilde{k}+1}}-\frac{\left(\left|V^{\prime}\right|^{p-2} V^{\prime}\right)^{\prime}+\frac{N-1}{y}\left|V^{\prime}\right|^{p-2} V^{\prime}}{(1-\delta t)^{(\tilde{k}+\tilde{r})(p-1)+\tilde{r}}} \\
& -\frac{\gamma V^{m}}{(1-\delta t)^{\tilde{k} m}}-\frac{\mu\left|V^{\prime}\right|^{r}}{(1-\delta t)^{r(\tilde{k}+\tilde{r})}}+\frac{\nu V^{\sigma}}{(1-\delta t)^{\tilde{k} \sigma}}.
\end{aligned}
$$
Next, we choose $t_0=t_0(p, m, r, \sigma, \delta, N, A)$ sufficiently close to $1 / \delta$, so that
$$
\begin{aligned}
P_1(v(x, t)) \leq & \frac{1}{(1-\delta t)^{\tilde{k}+1}}\left[\delta \tilde{k}\left(1+\frac{A}{\lambda}\right)+\frac{N}{A}\left(1-\delta t\right)^{1+\tilde{k}-\tilde{r}-(\tilde{k}+\tilde{r})(p-1)}\right. \\
& -\gamma-\mu\left(1-\delta t\right)^{\tilde{k}+1-r(\tilde{k}+\tilde{r})} \\
& \left.+\nu\left(1+\frac{A}{\lambda}\right)^{\sigma}\left(1-\delta t\right)^{\tilde{k}+1-\tilde{k} \sigma}\right] \leq 0.
\end{aligned}
$$
From the expression above, we conclude that $P_1(v(x,t))$ is negative since $\delta < \gamma \tilde{k}^{-1} \left(1 + \frac{A}{\lambda}\right)^{-1}$, and all exponents in the inequality remain positive under the following conditions:

\begin{equation} \label{positive_exponents}
\begin{aligned}
&1+\tilde{k}-\tilde{r}-(\tilde{k}+\tilde{r})(p-1) >0 \quad &&\text{since} \quad \tilde{r} < \frac{k + s + 1 - p}{p(k + s - 1) + N(p - 2)}, \\
&\tilde{k}+1-r(\tilde{k}+\tilde{r})>0 \quad &&\text{since} \quad \tilde{r}<\frac{k + s - r}{r(k + s - 1) + N(r - 1)}, \\
&\tilde{k}+1-\tilde{k} \sigma>0 \quad &&\text{since} \quad m > \max \{ p-1, r, \sigma \}.
\end{aligned}
\end{equation}

Now, we proceed to show that $P_2(v(x,t)) \leq 0$. To do so, we express $P_2(v(x,t))$ as follows:
$$P_2(v(x,t)) = -\alpha v^{k}\int_\Omega v^{s} dx +\beta v^{l} |\nabla v|^{q}. $$
Hence
\begin{equation*}
\begin{aligned}
& P_2(v(x,t))=-\frac{\alpha V^{k}(y)}{(1-\delta t)^{\tilde{k}(k+s)}} \int_{B\left(0, R(1-\delta t)^{\tilde{r}} \right)} V^{s}\left(\frac{|x|}{(1-\delta t)^{\tilde{r}}}\right) d x + \beta \frac{V^{l}(y)}{(1-\delta t)^{l \tilde{k}}} \\
& \times \frac{\left|V^{\prime}(y)\right|^{q}}{(1-\delta t)^{q(\tilde{r}+\tilde{k})}}.
\end{aligned}
\end{equation*}
By setting $K=\int_{B(0, R)} V^{s}(|\zeta|) d \zeta>0$, we obtain
$$
\int_{B\left(0, R(1-\delta t)^{\tilde{r}}\right)} V^{s}\left(\frac{|x|}{(1-\delta t)^{\tilde{r}}}\right) d x = \frac{K}{(1-\delta t)^{-N \tilde{r}}} .
$$
 For $t_0<t<\frac{1}{\delta}$, assuming that $t_0$ is sufficiently close to $\frac{1}{\delta}$, we have 
\begin{equation}
    P_2(v(x,t))\leq \frac{1}{(1-\delta t)^{\tilde{k}+1}}\left[-\alpha K+\beta \left(1+\frac{A}{\lambda}\right)^{l}(1-\delta t)^{-l \tilde{k}-q(\tilde{k}+\tilde{r})+\tilde{k}+1}\right] \leq 0.
\label{2.18}
\end{equation}
Therefore, $P_2(v(x,t))$ is also negative due to $\alpha$, $K>0$ and $-l \tilde{k}-q(\tilde{k}+\tilde{r})+\tilde{k}+1>0$ since
\begin{equation} \label{beta_term_positive}
\tilde{r} < \frac{k + s - l - q}{q(k + s - 1) + N(l + q - 1)}.
\end{equation}

Since $P_1(v(x,t))$ $\leq0$ and $P_2(v(x,t))$ $\leq$ $0$, we conclude that $P(v(x,t))$ $\leq$ $0$ in the case $0$ $\leq$ $y$ $\leq$ $A$. \\
Now, we consider the next case, where $A$ $\leq$ $y$ $\leq$ $R$.
Here, instead of decomposing the expression into parts, we evaluate $P(v(x,t))$ as a whole.
\begin{equation}
\begin{aligned}
P(v(x, t)) \leq & \frac{1}{(1-\delta t)^{\tilde{k}+1}}\left[\delta(\tilde{k}-\tilde{r} A)+\frac{N}{A}\left(1-\delta t\right)^{1+\tilde{k}-\tilde{r}-(\tilde{k}+\tilde{r})(p-1)}\right. \\
&+ \left(\frac{R}{A}\right)^{q(\lambda-1)} \beta(1-\delta t)^{-l \tilde{k}-q(\tilde{k}+\tilde{r})+\tilde{k}+1} \\
& -\mu\left(\frac{R}{A}\right)^{(\lambda-1) r}\left(1-\delta t\right)^{\tilde{k}+1-r(\tilde{k}+\tilde{r})} \\
& \left.+\nu\left(1-\delta t\right)^{\tilde{k}+1-\tilde{k} \sigma}\right] \leq 0.
\end{aligned}
\label{L_pv}
\end{equation}
The first term of this expression ensures that \( P(v(x, t)) \) is non-positive when $A$ $\leq$ $y$ $\leq$ $R$, since $A$ $>$ $\frac{\tilde{k}}{\tilde{r}}$ and $\delta > 0$. Owing to conditions \eqref{positive_exponents} and \eqref{beta_term_positive}, all the exponents in inequality \eqref{L_pv} remain positive.

Due to the embedding \( W_0^{1, \infty}(\Omega) \hookrightarrow C(\bar{\Omega}) \), it follows that there exists a constant \( C > 0 \) such that \( u_0(x) \geq C \) in a neighborhood \( B(0, \rho) \subset \subset \Omega \) for some \( \rho > 0 \). Since \( t_0 \) is chosen sufficiently close to \( \frac{1}{\delta} \), we can assume that $\overline{B\left(0, R\left(1-\delta t_0\right)^{\tilde{r}}\right)} \subset B(0, \rho)$. Additionally, we select \( C > 0 \) so that  
\[
u_0(x) \geq C \geq \frac{V(0)}{\left(1 - \delta t_0\right)^{\tilde{k}}} \geq v(x, t_0).
\]

Moreover, it is evident that \( v \leq  0 \) on \( \partial \Omega \times (t_0, \frac{1}{\delta}) \).  

Applying the comparison principle, we conclude that  
\[
u(x, t) \geq v(x, t + t_0), \quad x \in \Omega, \quad 0 < t < \frac{1}{\delta} - t_0.
\]

As a result, we obtain the upper bound \( T_{\max} \leq \frac{1}{\delta} - t_0 \).  
\end{proof}

\subsection{Global existence of the solution}
\begin{thm}\label{thm_glob_exis_1}
    Let $k$, $s$ $\geq$ 1 and $q$ $\geq \frac{p}{2}$, $\sigma$ $>m$. We assume that one of the following conditions hold: 
    \begin{enumerate} 
\item $q+l$ $>\max$ $\left\{p-1, k+s\right\}$;
\item $q+l$ $=k+s$ $>p-1$ and $|\Omega|$ $($or $\alpha)$ is small enough; 
\item $q+l$ $=p-1$ $>k+s$ and $\beta$ is large enough.  
\end{enumerate} 
    Then, the weak solution of the problem $(\ref{eq_1})$ with $\mu$ $=$ $0$ is uniformly bounded in the $L^\infty$ norm. 
\end{thm}

\begin{proof}[Proof of Theorem \ref{thm_glob_exis_1}] 
Let $\rho(\Omega)$ denote the diameter of $\Omega$. Since $\Omega$ is bounded, it follows that $\rho(\Omega)$ $<\infty$. Consider $\epsilon$ $\in$ $(0,1)$ such that there exists a ball of radius $\epsilon$ contained within $B(\cdot, \rho(\Omega)+1) \cap \Omega^c$. For any $a$ $\in$ $\Omega$, we choose a point $x_a$ satisfying the following condition
\begin{equation*}
    B(x_a, \epsilon) \subseteq B(x_a, \rho(\Omega)+1) \cup \Omega^c, \quad |x_a-a| < \rho(\Omega)+1.
\end{equation*}
We define the following function
\begin{equation*}
    v(x,t) = Le^{\tilde{\rho}} \quad \text{with} \ \tilde{\rho} = |x-x_a|, \ x \in \Omega.
\end{equation*}

Here, $L$ $\geq$ $1$ is a constant that will be determined later. Clearly, the inequality $\epsilon $ $\leq$ $\tilde{\rho}$  $\leq$  $\rho(\Omega) + 1$ holds. We now proceed to demonstrate that $v(x,t)$ serves as a super-solution to problem $(\ref{eq_1})$. By performing a direct computation, we obtain  
\[
\mathscr{L}_p v := v_t - \Delta_p v - \alpha v^k \int_\Omega v^s \, dx + \beta v^l |\nabla v|^q - \gamma v^m + \nu v^\sigma.
\]
Then \( V(x,t) \) satisfies  
\begin{equation*}
\begin{aligned}
    \mathscr{L}_{p}V = & - L^{p-1}e^{(p-1)\tilde{\rho}} \left(\frac{N-1}{\epsilon} + (p-1)\right) - \alpha L^{k} e^{k\tilde{\rho}} \int_\Omega L^{s} e^{s\tilde{\rho}} \, dx \\
    & + \beta L^{l+q} e^{\tilde{\rho}(l+q)}- \gamma L^{m} e^{m\tilde{\rho}} + \nu L^{\sigma} e^{\sigma \tilde{\rho}}.
\end{aligned}
\end{equation*}
To ensure that $\mathscr{L}_{p}V$ $\geq$ $0$, it is necessary to select an appropriate $L$ that satisfies 
\begin{equation*}
    \beta L^{l+q}e^{\tilde{\rho}(l+q)}+\nu L^{\sigma} e^{\sigma \tilde{\rho}}-\gamma L^{m} e^{m\tilde{\rho}} \geq \alpha L^{k} e^{k\tilde{\rho}} \int_\Omega L^{s} e^{s\tilde{\rho}} dx + L^{p-1}e^{(p-1)\tilde{\rho}}(\frac{N-1}{\epsilon} + p-1).
\end{equation*}
Next, we decompose this inequality into a system of two distinct inequalities.
\begin{equation}\label{syst_thm2.2}
\left\{
\begin{aligned}
    &\beta L^{l+q} e^ {\tilde{\rho} (l+q)} + \nu L^{\sigma} e^{\sigma \tilde{\rho}} - \gamma L^{m} e^{m \tilde{\rho}}  \geq 2\alpha L^{k} e^{k\tilde{\rho}} \int_\Omega L^{s} e^{s\tilde{\rho}}dx  \\
    &\beta L^{l+q} e^{\tilde{\rho}(l+q)}  + \nu L^{\sigma} e^{\sigma \tilde{\rho}} - \gamma L^{m} e^{m \tilde{\rho}}  \geq 2 L^{p-1}e^{(p-1) \tilde{\rho}} \left(\frac{N-1}{\epsilon} + (p-1)\right). 
\end{aligned}
\right.
\end{equation}

Now, we transform the first inequality of \eqref{syst_thm2.2} into the following system of inequalities:
\begin{equation*}
\left\{
\begin{aligned}
   & \beta L^{l+q} e^{\tilde{\rho} (l+q)} \geq 2\alpha L^{k} e^{k\tilde{\rho}} \int_\Omega L^{s} e^{s\tilde{\rho}}dx \\&
    \nu L^{\sigma} e^{\sigma \tilde{\rho}} \geq \gamma L^{m} e^{m \tilde{\rho}}.    
\end{aligned}
\right.
\end{equation*}
Since
\begin{align*}
     2 \alpha L^{k} e^{k\tilde{\rho}} \int_\Omega L^{s} e^{s\tilde{\rho}}dx = 2 \alpha C L^{s} |\Omega| e^{k\tilde{\rho}},
\end{align*}
where $C=e^{s(\rho(\Omega)+1)}.$
We have  
\[
\left\{
\begin{array}{l}
    \beta L^{l+q-k-s} e^{\tilde{\rho} (l+q)} \geq 2 \alpha C  |\Omega| e^{k\tilde{\rho}}\\[10pt]
    \nu L^{\sigma - m} e^{\sigma \tilde{\rho}} \geq \gamma e^{m \tilde{\rho}}.  
\end{array}
\right.
\]

Let us now divide the second inequality of \eqref{syst_thm2.2} into the following system of inequalities:
\[
\left\{
\begin{array}{l}
\beta L^{l+q} e^{\tilde{\rho} (l+q)} \geq 2 L^{p-1}e^{(p-1) \tilde{\rho}}(\frac{N-1}{\epsilon} + (p-1) \\ 
\nu L^{\sigma} e^{\sigma \tilde{\rho}} \geq \gamma L^{m} e^{m \tilde{\rho}}.  
\end{array}
\right.
\]
Let us express the constant $L$ from the given system and obtain the following transformed system of inequalities:
\[
\left\{
\begin{array}{l}
\beta L^{l+q-p+1} e^{\tilde{\rho} (l+q)} \geq 2 e^{(p-1) \tilde{\rho}} \left( \frac{N-1}{\epsilon} + (p-1) \right) \\
\nu L^{\sigma - m} e^{\sigma \tilde{\rho}} \geq \gamma e^{m \tilde{\rho} }.
\end{array}
\right.
\] 
\begin{enumerate}
    \item If $q+l >$ max $\{p-1, k+s\}$, then the constant $L$ must satisfy
    \begin{multline*}
    L \geq \max \left\{ \left[ \frac{2}{\beta} \left( \frac{N-1}{\epsilon} + (p-1) \right) \right]^{\frac{1}{l+q-p+1}}, 
    \left[ \frac{2\alpha C |\Omega|}{\beta} \right]^{\frac{1}{l+q-k-s}}, 
    \left[ \frac{\gamma}{\nu} \right]^{\frac{1}{\sigma - m}} \right\}.
\end{multline*}
In order to guarantee that $v(x,0)\geq u_0$, we must also ensure that $L\geq ||u_0||_\infty$. Then, we select $L$ such that
\begin{align*}
     L = \max \left\{ \left[ \frac{2}{\beta} \left( \frac{N-1}{\epsilon} + (p-1) \right) \right]^{\frac{1}{l+q-p+1}}, 
    \left[ \frac{2\alpha C |\Omega|}{\beta} \right]^{\frac{1}{l+q-k-s}}, 
    \left[ \frac{\gamma}{\nu} \right]^{\frac{1}{\sigma - m}},1,||u_0||\infty \right\}.
\end{align*}

\item If $q+l = k+s > p-1$, we need to take $L$ 
\begin{align*}
    L = \max \left\{ \left[ \frac{2}{\beta} \left( \frac{N-1}{\epsilon} + (p-1) \right) \right]^{\frac{1}{l+q-p+1}}, \left[ \frac{\gamma}{\nu} \right]^{\frac{1}{\sigma - m}}, 1, ||u_0||_\infty \right\},
\end{align*}  
And we select the Lebesgue measure of $\Omega$ (denoted as $|\Omega|$) to satisfy the following condition
\[
|\Omega| \leq \frac{\beta}{2\alpha C} \quad \text{or} \quad \alpha \leq \frac{\beta}{2C|\Omega|}
\]
\item If $q+l = p-1 > k+s$, it is necessary to choose
\[
L = \max \left\{ \left[ \frac{2\alpha C |\Omega|}{\beta} \right]^{\frac{1}{l+q-k-s}}, \left[ \frac{\gamma}{\nu} \right]^{\frac{1}{\sigma - m}}, 1, ||u_0||_\infty \right\},
\] 
and we determine the measure of $\beta$ such that 
\[
\beta \geq 2 \left( \frac{N-1}{\epsilon} + (p-1) \right).
\]
\end{enumerate}
It is evident that $v(x,t)\geq u(x,t)=0$ on $\partial \Omega$. 
Therefore, by applying the comparison principle, we deduce that $v(x,t)$ serves as a super-solution to problem (\ref{eq_1}). As a result, the following inequality holds:
$$0\leq u(x,t)\leq v(x,t)\leq Le^{\rho(\Omega)+1}.$$
\end{proof}

\end{document}